\documentclass[12pt,reqno]{amsart}
\usepackage{fullpage}
\usepackage{amsmath}
\usepackage{mathrsfs,amssymb,graphicx,verbatim,hyperref}
\usepackage{paralist}

\renewcommand{\setminus}{\smallsetminus}

\usepackage{ifthen}

\newcommand{\ifsodaelse}[2]{\ifthenelse{\isundefined{\SODAF}}{#2}{#1}}

\ifsodaelse    {\usepackage{ltexpprt}\usepackage{amsmath}}{}
\usepackage{mathrsfs,amssymb,graphicx,verbatim}
\usepackage{paralist}

\newcommand\remove[1]{}

\newcommand{\rnote}[1]{}
\newcommand{\jnote}[1]{}

\newcommand{\1}{\mathbf{1}}
\newcommand{\e}{\varepsilon}
\newcommand{\R}{\mathbb{R}}

\newcommand{\E}{\mathbb{E}}

\newcommand{\N}{\mathbb{N}}

\newcommand\F{{{\mathscr F}}}

\newtheorem{theorem}{Theorem}[section]
\newtheorem{lemma}[theorem]{Lemma}

\newtheorem{corollary}[theorem]{Corollary}

\newtheorem{remark}{Remark}[section]

\newcommand{\eqdef}{\stackrel{\mathrm{def}}{=}}
\date{}

\renewcommand{\le}{\leqslant}
\renewcommand{\ge}{\geqslant}

  \newtheorem{proposition}[subsection]{Proposition}

\include{psfig}

\begin{document}

\title{On the Banach space valued Azuma inequality and small set isoperimetry of Alon-Roichman graphs}

\author{Assaf Naor}
\address{Courant Institute, 251 Mercer Street, New York, NY 10012, USA.}
\email{naor@cims.nyu.edu}
\thanks{Research supported in part by NSF grants CCF-0635078 and CCF-0832795, BSF grant 2006009, and the Packard Foundation.}

\date{}
\maketitle

\begin{abstract}
We discuss the connection between the expansion of small sets in
graphs, and the Schatten norms of their adjacency matrix. In
conjunction with a variant of the Azuma inequality for uniformly
smooth normed spaces, we deduce improved bounds on the small set
isoperimetry of Abelian Alon-Roichman random Cayley graphs.
\end{abstract}

\section{Introduction}

In what follows all graphs are allowed to have multiple edges and
self loops. For a finite group $\Gamma$ of cardinality $n$, the
Cayley graph associated to the  group elements $g_1,\ldots,g_k\in
\Gamma$ is the $2k$-regular graph $G=(V,E)$, where $V=\Gamma$, and
the number of edges joining $u,v\in \Gamma$ equals $\left|\{i\in
\{1,\ldots,n\}:\ uv^{-1}=g_i\}\right|+|\{i\in \{1,\ldots,n\}:\
uv^{-1}=g_i^{-1}\}|$.

For a graph $G=(V,E)$ and $S\subseteq V$, let $E_G(S,V\setminus S)$
denote the size of the edge boundary of $S$, i.e., the number of
edges joining $S$ and its complement. The Alon-Roichman
theorem~\cite{AR94} asserts that random Cayley graphs obtained by
choosing $k$ group elements independently and uniformly at random
are good expanders, provided $k$ is large enough:

\begin{theorem}[Alon-Roichman theorem]\label{thm:AR}
For every $\e\in (0,1)$ there exists $c(\e)\in (0,\infty)$ with the
following property. Let $\Gamma$ be a finite group of cardinality
$n$. Assume that $k\ge c(\e)\log n$. Then with probability at least
$\frac12$ over $g_1,\ldots,g_k$ chosen independently and uniformly
at random from $\Gamma$, if $G$ is the Cayley graph associated to
$g_1,\ldots,g_k$, then for every $\emptyset\neq S\subsetneq \Gamma$
we have
$$
\left|\frac{E_G(S,V\setminus S)}{\frac{2k}{n}|S|(n-|S|)}-1\right|\le
\e.
$$
\end{theorem}
Subsequent investigations by several
authors~\cite{Pak99,LR04,LS04,WX08,CM08} yielded new proofs, with
various improvements, of the Alon-Roichman theorem. The best known
upper bound on $c(\e)$ is $O(1)/\e^2$; see~\cite{CM08} for the best
known estimate on the implied constant in the $O(1)$ term.

Here we obtain an improved estimate on the isoperimetric profile of
random Cayley graphs of Abelian groups:
\begin{theorem}\label{thm:our} There exists a universal constant $c\in (0,\infty)$
with the following property. Let $\Gamma$ be an Abelian group of
cardinality $n$. Assume that $k\ge \frac{c\log n}{\e^2}$. Then with
probability at least $\frac12$ over $g_1,\ldots,g_k$ chosen
independently and uniformly at random from $\Gamma$, if $G$ is the
Cayley graph associated to $g_1,\ldots,g_k$, then for every
$S\subseteq \Gamma$ with $2\le |S|\le \frac{n}{2}$ we have
\begin{equation}\label{eq:sqrt}
\left|\frac{E_G(S,V\setminus S)}{\frac{2k}{n}|S|(n-|S|)}-1\right|\le
\e\sqrt{\frac{\log |S|}{\log n}}.
\end{equation}
\end{theorem}
The possible validity of an estimate such as~\eqref{eq:sqrt} for
finite groups $\Gamma$ that are not necessarily Abelian remains an
interesting open question.


When $k=o(\log n)$, the graphs $G$ of Theorem~\ref{thm:our} need not be connected, and they are never expanders~\cite[Prop.~3]{AR94}. Nevertheless, with positive probability, sufficiently small sets in such graphs do have a large edge boundary:
\begin{theorem}\label{thm:small k}
There exists a universal constant $c\in (0,\infty)$ with the
following property. Let $\Gamma$ be an Abelian group of cardinality
$n$. Fix $k\in \N$ and $\e\in (0,1)$. Then with probability at least
$\frac12$ over $g_1,\ldots,g_k$ chosen independently and uniformly
at random from $\Gamma$, if $G$ is the Cayley graph associated to
$g_1,\ldots,g_k$, then for every $\emptyset \neq S\subsetneq \Gamma$
we have
\begin{equation*}
 |S|\le e^{c\e^2k}\implies \left|\frac{E_G(S,V\setminus S)}{\frac{2k}{n}|S|(n-|S|)}-1\right|\le
\e.
\end{equation*}
\end{theorem}
The main purpose of this note is {\em not} to obtain results such as
Theorem~\ref{thm:our} and Theorem~\ref{thm:small k}, even
if these have independent interest. Our goal is rather to present
a method to prove expansion of small sets,
going beyond the standard spectral gap techniques. In
addition, we highlight a simple and general geometric argument that
allows one to reason about such questions for random objects like
Alon-Roichman graphs. The rest of this introduction will therefore
be devoted to a description these issues. We note that a draft of this manuscript
has been circulating for several years, and we were motivated
to make it publicly available since the ideas presented here
inspired recent progress in theoretical computer science;
see~\cite{ABS10}.

\subsection{Schatten bounds and small set expansion}

For an $n$-vertex graph $G=(V,E)$ and $u,v\in V$, let $e(u,v)$ denote the number of edges joining $u$ and $v$ if $u\neq v$, and twice the number of self loops at $u$ if $u=v$. If $G$ is $d$-regular, then the normalized adjacency matrix of $G$ is the $n\times n$ matrix $A(G)$ whose entry at $(u,v)\in V\times V$ equals $e(u,v)/d$. We will denote by $1=\lambda_1(G)\ge \lambda_2(G)\ge \cdots \ge
\lambda_n(G)$ the decreasing rearrangement of eigenvalues of
$A(G)$.

The well-established connection, due to~\cite{AM85,Tan84} (see also~\cite[Thm.~9.2.1]{AS00}), between spectral gaps and graph expansion, reads as follows: for every $\emptyset \neq S\subsetneq V$ we have
\begin{equation}\label{eq:mixing}
\left|\frac{E_G(S,V\setminus S)}{\frac{d}{n}|S|(n-|S|)}-1\right|\le
\max_{i\in \{2,\ldots,n\}}|\lambda_i(G)|.
\end{equation}

Let $L_2(V)$ denote the vector space $\R^V$, equipped with the scalar product
$$
\forall\ x,y:V\to \R,\quad \langle x,y\rangle
\stackrel{\mathrm{def}}{=} \frac{1}{n}\sum_{u\in V}x(u)y(u).
$$
The following lemma is a natural variant of the
bound~\eqref{eq:mixing}.
\begin{lemma}\label{lem:interpolated infty}
Fix $p\in [1,\infty]$. Assume that $L_2(V)$ has an orthonormal
eigenbasis of $A(G)$ consisting of vectors all of whose entries are
bounded by $1$ in absolute value. Then for every $\emptyset\neq
S\subsetneq V$ we have,
$$
\left|\frac{E(S,V\setminus S)}{\frac{d}{n}|S|(n-|S|)}-1\right|\le
\left(\sum_{i=2}^n|\lambda_i(G)|^p\right)^{\frac{1}{p}}\left(\frac{|S|^{\frac{1}{p+1}}(n-|S|)^{\frac{p}{p+1}}+
|S|^{\frac{p}{p+1}}(n-|S|)^{\frac{1}{p+1}}}{n}\right)^{\frac{p+1}{p}}.
$$
\end{lemma}
See Lemma~\ref{lem:interpolated} below for a more general version of Lemma~\ref{lem:interpolated infty}, which does not require the existence of an eigenbasis with good $L_\infty$ bounds. We chose to state the above simpler version of Lemma~\ref{lem:interpolated} in the introduction, since the assumption of Lemma~\ref{lem:interpolated infty} holds automatically for Cayley graphs of Abelian groups, the eigenbasis being the characters of the group.

Alon and Roichman~\cite{AR94}, as well as the subsequent
work~\cite{Pak99,LR04,LS04,WX08,CM08}, proved Theorem~\ref{thm:AR}
by showing that, under the assumptions appearing in the statement of
Theorem~\ref{thm:AR}, we have $\max_{i\in
\{2,\ldots,n\}}|\lambda_i(G)|\le \e$, and then appealing
to~\eqref{eq:mixing}. We will prove Theorem~\ref{thm:our} and
Theorem~\ref{thm:small k} by applying Lemma~\ref{lem:interpolated
infty} with an appropriate choice of $p$ (depending on the
cardinality of $S$). With this goal in mind, we need to be able to
argue about the quantity
$\left(\sum_{i=2}^n|\lambda_i(G)|^p\right)^{1/p}$ when $G$ is the
random graph appearing in Theorem~\ref{thm:our}. This can be done
via simple geometric considerations from Banach space theory.

\subsection{Banach space valued concentration}

The singular values of an $n\times n$ matrix $A$, i.e., the
eigenvalues of $\sqrt{A^*A}$, are denoted $s_1(A)\ge
s_2(A)\ge\cdots\ge s_n(A)$. For $p\in [1,\infty]$, the Schatten
$p$-norm of $A$, denoted $\|A\|_{S_p}$, is defined to be
$\left(\sum_{i=1}^ns_i(A)^p\right)^{1/p}$. Thus, the quantity
$\left(\sum_{i=2}^n|\lambda_i(G)|^p\right)^{1/p}$ appearing in
Lemma~\ref{lem:interpolated infty} equals $
\left\|\left(I-\frac{1}{n}J\right)A(G)\right\|_{S_p}, $ where $I$ is
the $n\times n$ identity matrix and $J$ is the $n\times n$ matrix
all of whose entries are $1$.

Fix a group $\Gamma$ of cardinality $n$. Let $R:\Gamma\to
GL(L_2(\Gamma))$ be the right regular representation, i.e.,
$(R(g)\phi)(h)=\phi(gh)$ for every $\phi:\Gamma\to \R$ and $g,h\in
\Gamma$. The normalized adjacency matrix of the Cayley graph associated to $g_1,\ldots,g_k\in \Gamma$ is given by
\begin{equation}\label{eq:def A(g_i)}
A(g_1,\ldots,g_k)\eqdef \frac{1}{2k}\sum_{i=1}^k
\left(R(g_i)+R(g_i^{-1})\right).
\end{equation}
In order to apply Lemma~\ref{lem:interpolated infty}, we are therefore interested in the random quantity
\begin{equation}\label{eq:the S_p norm}
\left\|\sum_{i=1}^k\left(I-\frac{1}{n}J\right)\frac{R(g_i)+R(g_i^{-1})}{2k}\right\|_{S_p}.
\end{equation}

All the known proofs of the Alon-Roichman theorem, corresponding to
the case $p=\infty$ in~\eqref{eq:the S_p norm},  proceed by proving
the desired deviation inequality for operator norm valued random
variables; the original proof of Alon and Roichman uses the Wigner
semicircle method, while later proofs rely on the Ahlswede-Winter
matrix valued deviation bound~\cite{AW02}. Alternatively, one can
use in this context the moment inequalities arising from the
non-commutative Khinchine inequalities of Lust-Piquard and
Pisier~\cite{Lust86,L-PP91}, and this method also yields the
inequalities that we need for the deviation of the Schatten $p$-norm
in~\eqref{eq:the S_p norm}. Nevertheless, all of these approaches
are specific to operator-valued random variables, and are deeper
than the simple argument that we present below. It turns out that
for our purposes, it suffices to use an elementary geometric
argument that ignores the specific structure of matrix spaces---it
works for random variables taking values in arbitrary uniformly
smooth normed spaces, of which the Schatten $p$-norms are a special
case.

For a Banach space $(X,\|\cdot\|)$, the triangle inequality implies
that $\|x+\tau y\|+\|x-\tau y\|\le 2+ 2\tau$ for every two unit
vectors $x,y\in X$ and every $\tau>0$. $X$ is said to be uniformly
smooth if $\|x+\tau y\|+\|x-\tau y\|\le 2+ o(\tau)$, where the
$o(\tau)$ term is independent of the choice of unit vectors $x,y\in
X$. Formally, consider the following quantity, called the modulus of
uniform smoothness of $X$.
\begin{equation}\label{eq:def rho}
\rho_X(\tau)\stackrel{\mathrm{def}}{=}\sup\left\{\frac{\|x+\tau
y\|+\|x-\tau y\|}{2}-1:\ x,y\in X,\ \|x\|=\|y\|=1\right\}.
\end{equation}
Then $X$ is uniformly smooth if $\lim_{\tau\to
0}\frac{\rho_X(\tau)}{\tau}=0$.

$X$ is said to have a modulus of smoothness of power type $2$ if
there exists $s>0$ such that for all $\tau>0$ we have
$\rho_X(\tau)\le s\tau^2$. For simplicity, we will only deal here
with spaces that have a modulus of smoothness of power type $2$. All
of our results below carry over, with obvious modifications, to
general uniformly smooth spaces (of course, in this more general
setting, the probabilistic bounds that we get will no longer be
sub-Gaussian). For concreteness, if $p\ge 2$ and $S_p$ denotes the
space of all $n\times n$ matrices equipped with the Schatten
$p$-norm, then for every $\tau>0$ we have $\rho_{S_p}(\tau)\le
\frac{p-1}{2}\tau^2$. The fact that the modulus of smoothness of
$S_p$ has power type $2$ when $p\ge 2$ was first proved by
Tomczak-Jaegermann in~\cite{TJ74}. The exact modulus of smoothness
of $S_p$ was computed in~\cite{BCN94}.
The case of $L_p(\mu)$
spaces is much older---the modulus of smoothness in this case was
computed by Hanner in~\cite{Han56}.

An Azuma-type deviation inequality holds for general norms whose
modulus of smoothness has power type $2$.

\begin{theorem}\label{thm:azuma} There exists a universal constant $c\in (0,\infty)$ with the following property. Fix
$s>0$ and assume that a Banach space $(X,\|\cdot\|)$ satisfies
$\rho_X(\tau)\le s\tau^2$ for all $\tau>0$. Fix also a sequence of
positive numbers $\{a_k\}_{k=1}^\infty\subseteq (0,\infty)$. Let
$\{M_k\}_{k=1}^\infty\subseteq X$ be an $X$-valued martingale
satisfying the pointwise bound $\|M_{k+1}-M_k\|\le a_k$ for all
$k\in \N$. Then for every $u>0$ and $k\in \N$ we have
\begin{equation}\label{eq:azuma e^s}
\Pr\left[\|M_{k+1}-M_1\|\ge u\right]\le e^{s+2} \cdot
\exp\left(-\frac{cu^2}{a_1^2+\cdots+a_k^2}\right).
\end{equation}
\end{theorem}
Theorem~\ref{thm:azuma} is a consequence of well understood moment
inequalities in Banach space theory. The key insights here are due
to the work of Pisier. Theorem~\ref{thm:azuma} relies on an estimate
of implicit constants appearing in Pisier's inequality~\cite{Pis75};
this is done in Section~\ref{sec:azuma}. While these bounds are not
available in~\cite{Pis75}, undoubtedly Pisier could have computed
them if they were needed for the purpose of his investigations
in~\cite{Pis75}. Therefore, Section~\ref{sec:azuma} should be viewed
as mainly expository. In addition to obtaining the estimates that we
need, another purpose of Section~\ref{sec:azuma} is to present the
proof in a way which highlights the clarity and simplicity of the
general geometric argument leading to Theorem~\ref{thm:azuma}.

Note that the exponential dependence on $s$ in~\eqref{eq:azuma e^s}
cannot be improved. A roundabout way to see this is to note that
when $X$ is the space of $n\times n$ matrices equipped with the
Schatten $p$ norm, then since in this case $s\le p/2$, and for
$p=\log n$ we have $\|\cdot\|_{S_p}\asymp \|\cdot\|_{S_\infty}$ ($=$
the operator norm), the inequality~\eqref{eq:azuma e^s} corresponds
(for this value of $p$) to the Ahlswede-Winter deviation inequality
used in~\cite{LR04,CM08,WX08} to prove the (sharp) logarithmic
dependence of $k$ on $n$ in the Alon-Roichman theorem. For random
variables taking values in the space of matrices equipped with the
operator norm, deeper methods lead to results that are more refined
than Theorem~\ref{thm:azuma}, but do not have an interpretation in
the setting of general uniformly smooth Banach spaces, and are not
needed for our purposes; see~\cite{Tro10} for more information on
this topic.

\section{The Azuma inequality in uniformly smooth normed
spaces}\label{sec:azuma}

Our main result is the following theorem.

\begin{theorem}\label{thm:moment} Fix $s>0$ and $q\ge 2$. Assume that a Banach space $X$ satisfies
$\rho_X(\tau)\le s\tau^2$ for all $\tau>0$. Let
$\{Z_n\}_{n=1}^\infty$ be $X$-valued random variables such that for
all $n\in N$ we have $\E\left[\|Z_n\|^q\right]<\infty$. Denote
$S_n=Z_1+\cdots+Z_n$. Assume that for all $n\in \N$ we have
\begin{equation}\label{eq:geometric condition}
\E\left[\|S_n-Z_{n+1}\|^q\right]\ge \E\left[\|S_n\|^q\right].
\end{equation} Then
\begin{equation}\label{eq:moment bound}
\Big(\E\left[\|S_n\|^q\right]\Big)^{1/q}\le 8\sqrt{s+q}\cdot
\sqrt{\sum_{j=1}^n\left(\E\left[\|Z_j\|^q\right]\right)^{2/q}}.
\end{equation}
\end{theorem}

Before proving Theorem~\ref{thm:moment} we record two concrete
examples. The most important case is when $\{Z_n\}_{n=1}^\infty$
form a martingale difference sequence. In other words, there exists
a filtration $\F_1\subseteq \F_2\subseteq \cdots$ such that
$Z_1,\ldots,Z_n$ are measurable with respect to $\F_{n}$ for all
$n\in \N$, and for $m>n$ we have $\E\left[Z_m\big|\F_n\right]=0$. In
this case, using the notation of Theorem~\ref{thm:moment} and the
convexity of the norm, we see that
\begin{equation}\label{eq:to take}
\E\left[\|S_n-Z_{n+1}\|^q\big|\F_n\right]\ge
\left\|\E\left[S_n-Z_{n+1}\big|\F_n\right]\right\|^q=\|S_n\|^q.
\end{equation}
By taking expectation we get that the assumption~\eqref{eq:geometric
condition} is satisfied.

Another example worth mentioning is when $p$ is an even integer and
$Z_n\in L_p(\Omega,\mu)$ satisfy the point-wise condition $\E
\left[Z_{n+1}S_n^{p-1}\right]\le 0$ for all $n\in \N$. In this case
$$
\E \left[\|S_n-Z_{n+1}\|_p^p\right] =\E
\left[\int_\Omega(S_n-Z_{n+1})^pd\mu\right]\ge \E \left[\int_\Omega
\left(S_n^p-pZ_{n+1}S_n^{p-1}\right)d\mu\right]\ge
\E\left[\|S_n\|_p^p\right].
$$
Therefore the assumption~\eqref{eq:geometric condition} holds for
$q=p$.

\begin{proof}[Proof of Theorem~\ref{thm:azuma}]
An application of Theorem~\ref{thm:moment} to $Z_n=M_{n+1}-M_n$,
together with Markov's inequality, shows that for all $q\ge
\max\{2,s\}$ and $u>0$,
\begin{equation}\label{eq:for optimal q}
\Pr\left[\|M_{k+1}-M_1\|\ge u\right]\le
\left(\frac{16\sqrt{q(a_1^2+\cdots+a_k^2)}}{u}\right)^q.
\end{equation}
The optimal choice of $q$ in~\eqref{eq:for optimal q} is
$q=\frac{u^2}{256e(a_1^2+\cdots+a_k^2)}$, which is an allowed value
of $q$ provided $u^2\ge \max\{s,2\}\cdot 256e(a_1^2+\cdots+a_k^2)$.
This implies~\eqref{eq:azuma e^s}.
\end{proof}


We now pass to the proof of Theorem~\ref{thm:moment}. We start with
the following lemma, whose proof is a slight variant of the proof of
Proposition 7 in~\cite{BCN94}.

\begin{lemma}\label{lem:two point} Assume that $\rho_X(\tau)\le s\tau^2$ for all
$\tau>0$. Then for every $x,y\in X$ and for every $q\ge 2 $ we
have
\begin{eqnarray}\label{eq:beta}
\frac{\|x+y\|^q+\|x-y\|^q}{2}\le \left(\|x\|^2+5(s+q)\cdot
\|y\|^2\right)^{q/2}.
\end{eqnarray}
\end{lemma}

\begin{proof} Assume first of all that $\|x\|=1$ and $\|y\|\le 1$.
Denote
$$b\stackrel{\mathrm{def}}{=} \frac{\|x+y\|+\|x-y\|}{2}\quad
\mathrm{and}\quad \beta\stackrel{\mathrm{def}}{=}
\frac{\|x+y\|-\|x-y\|}{\|x+y\|+\|x-y\|}. $$ Then $b\ge 1$, since the
function $\tau\mapsto \|x+\tau y\|+\|x-\tau y\|$ is convex and even,
and hence attains its minimum at $\tau=0$. Also, the triangle
inequality implies that $\beta \le \frac{\|y\|}{b}\le \|y\|\le 1.$
If we write
$$\left(\frac{(1+\beta)^q+(1-\beta)^q}{2}\right)^{1/q}=1+\theta$$ then
$\theta\in [0,1]$ and $\theta\le (q-1)\beta^2$. Both of these
inequalities are elementary numerical facts; the  proof of the
latter inequality can be found in many places,
e.g.,~\cite[Lem.~1.e.14]{LT79}.

Now,
\begin{multline*}
\left(\frac{\|x+y\|^q+\|x-y\|^q}{2}\right)^{1/q}
=b\left(\frac{(1+\beta)^q+(1-\beta)^q}{2}\right)^{1/q}\stackrel{(\clubsuit)}{\le}
\left[1+\rho_X(\|y\|)\right]\cdot
[1+\theta]\\\stackrel{(\clubsuit\clubsuit)}{\le}
\sqrt{1+5\left[\rho_X(\|y\|)+\theta\right]}\le
\sqrt{1+5(s+q)\|y\|^2},
\end{multline*}
where in $(\clubsuit)$ we used the definition of $\rho_X$ and
$\theta$ and in $(\clubsuit\clubsuit)$ we used the elementary
inequality $(1+u)(1+v)\le \sqrt{1+5(u+v)}$, which is valid for all
$u,v\in [0,1]$. This proves the assertion of Lemma~\ref{lem:two
point} when $\|y\|\le \|x\|$. When $\|y\|\ge \|x\|$ apply the same
reasoning with the roles of $x$ and $y$ reversed, and use the bound
$\|y\|^2+5(s+q)\cdot \|x\|^2\le \|x\|^2+ 5(s+q)\cdot \|y\|^2$.
\end{proof}

 In what follows, if $(\Omega,\mu)$ is a
measure space then $L_q(X,\Omega,\mu)$ will denote the Banach
space of all functions $f:\Omega\to X$ such that
$$
\|f\|_{L_q(X,\Omega,\mu)}^q\stackrel{\mathrm{def}}{=} \int_\Omega
\|f\|^qd\mu<\infty. $$
In~\cite{Fig76} it is shown that if $q\ge 2$ and $X$ is $2$-smooth
then $L_q(X,\Omega,\mu)$ is also $2$-smooth. The dependence of the
modulus of smoothness of $L_q(X,\Omega,\mu)$ on $q$ and the modulus
of smoothness of $X$ can be deduced from the proofs in~\cite{Fig76},
but is not stated there explicitly. This dependence is crucial for
us, so we will now show how it easily follows from
Lemma~\ref{lem:two point}.

\begin{corollary}\label{coro:L_q} Assume that $\rho_X(\tau)\le s\tau^2$ for all
$\tau>0$. Then
 $\rho_{L_q(X,\Omega,\mu)}(\tau)\le 4(s+q)\tau^2$  for every $\tau>0$.
\end{corollary}

\begin{proof}
Fix $f,g\in L_q(X,\Omega,\mu)$ with
$\|f\|_{L_q(X,\Omega,\mu)}=\|g\|_{L_q(X,\Omega,\mu)}=1$ and
$\tau>0$. Then,
\begin{eqnarray}\label{eq:use2point}
\left(\frac{\|f+\tau g\|_{L_q(X,\Omega,\mu)}+\|f-\tau
g\|_{L_q(X,\Omega,\mu)}}{2}\right)^q&\le& \frac{\|f+\tau
g\|_{L_q(X,\Omega,\mu)}^q+\|f-\tau
g\|^q_{L_q(X,\Omega,\mu)}}{2}\nonumber\\&=&\int_\Omega
\frac{\|f+\tau g\|^q+\|f-\tau
g\|^q}{2}d\mu\nonumber\\&\stackrel{(*)}{\le}&\int_\Omega\left(\|f\|^2+5(s+q)\cdot
\tau^2\|g\|^2\right)^{q/2}d\mu\nonumber\\\nonumber
&\stackrel{(**)}{\le}&
\left(\|f\|_{L_q(X,\Omega,\mu)}^2+5(s+q)\tau^2\|g\|_{L_q(X,\Omega,\mu)}^2\right)^{q/2}\\
&=& \left(1+5(s+q)\tau^2\right)^{q/2},
\end{eqnarray}
where in $(*)$ we used Lemma~\ref{lem:two point} and in $(**)$ we
used the triangle inequality in $L_{q/2}(\Omega,\mu)$. It follows
from~\eqref{eq:use2point} that
\begin{equation*}
\frac{\|f+\tau g\|_{L_q(X,\Omega,\mu)}+\|f-\tau
g\|_{L_q(X,\Omega,\mu)}}{2}\le \sqrt{1+5(s+q)\tau^2}\le
1+4(s+q)\tau^2.\qedhere
\end{equation*}
\end{proof}

The following lemma goes back to~\cite{Lin63,BG69} (see also
Proposition 2.2 in~\cite{Pis75}).

\begin{lemma}\label{lem:induction}
Assume that $\rho_X(\tau)\le s\tau^2$ for all $\tau>0$. Let
$\{x_n\}_{n=1}^\infty\subseteq X$ be a sequence of vectors in $X$
and for every $n\in \mathbb N$ denote $S_n=x_1+\cdots+x_n$. Assume
that for all $n\in \N$ we have $\|S_n-x_{n+1}\|\ge \|S_n\|$. Then
for every $n\in \N$,
$$
\|S_n\|^2\le 10(s+2)\left(\|x_1\|^2+\cdots+\|x_n\|^2\right).
$$
\end{lemma}

\begin{proof} Apply Lemma~\ref{lem:two point} with $q=2$, $x=S_n$
and $y=x_{n+1}$ to get that
\begin{eqnarray}\label{eq:almost}
\frac{\|S_{n+1}\|^2+\|S_n\|^2}{2}\stackrel{(\star)}{\le}\frac{\|S_n+x_{n+1}\|^2+\|S_n-x_{n+1}\|^2}{2}\le
\|S_n\|^2+5(s+2)\|x_{n+1}\|^2,
\end{eqnarray}
where in $(\star)$ we used the assumption $\|S_n-x_{n+1}\|\ge
\|S_n\|$. Inequality~\eqref{eq:almost} is equivalent to
$\|S_{n+1}\|^2\le \|S_n\|^2+10(s+2)\|x_{n+1}\|^2$. Therefore
Lemma~\ref{lem:induction} follows by induction.
\end{proof}

\begin{proof}[Proof of Theorem~\ref{thm:moment}]
Lemma~\ref{lem:induction}, combined with Corollary~\ref{coro:L_q},
implies Theorem~\ref{thm:moment}.
\end{proof}

\begin{remark}
{\em Readers who are mainly interested in the case of operator
valued random variables should note that the above argument proves
the general Azuma inequality~\eqref{eq:azuma e^s} in a
self-contained way, except that we quoted the fact that
$\rho_{S_p}(\tau)\lesssim p\tau^2$. We wish to stress that the proof
of this fact is elementary and accessible to non-experts. When
$p=2k$ is an even integer, the sharp estimate on the modulus of
smoothness of $S_p$ was proved in~\cite{TJ74} by expending the
quantity
$\|A+B\|_{S_p}^p+\|A-B\|_{S_p}^p=\mathrm{trace}\left([(A+B)(A^*+B^*)]^k+[(A-B)(A^*-B^*)]^k\right)$,
and estimating the summands in the resulting sum separately. When
$p\ge 2$ is not an even integer, the computation of the sharp
modulus of smoothness of $S_p$ in~\cite{BCN94} is more subtle, but
still elementary. Note that for the purpose of our application to
small set expansion of graphs, the case of even $p$ suffices. Also,
in the case of $S_p$, the above proof is much shorter (yielding
better constants), since the intermediate steps of
Lemma~\ref{lem:two point} and Corollary~\ref{coro:L_q} are not
needed---the inequalities obtained in Lemma~\ref{lem:two point} and
Corollary~\ref{coro:L_q} are the way that $\rho_{S_p}(\tau)$ was
estimated in~\cite{TJ74,BCN94} in the first place. The role of these
elementary intermediate steps is only to relate the standard
definition~\eqref{eq:def rho} of the modulus of uniform smoothness
to inequalities such as~\eqref{eq:beta}, but in the literature, when
one estimates $\rho_{X}(\tau)$, this is often done by
proving~\eqref{eq:beta} directly.}
\end{remark}

\section{Schatten norm bounds and graph expansion}

Let $G=(V,E)$ be an $n$-vertex $d$-regular graph. For $p\ge 1$, the
normed space $L_p(V)$ is the space of all $x\in \R^V$, equipped with
the norm
$$\|x\|_{L_p(V)}\stackrel{\mathrm{def}}{=}
\left(\frac{1}{n}\sum_{u\in V} |x(u)|^p\right)^{1/p}.
$$

In what follows, whenever we refer to an orthornormal  eigenbasis
$e_1,\ldots,e_n$ of the normalized adjacency matrix $A(G)$, it will
always be understood that it is orthonormal in $L_2(V)$,  the
eigenvectors are indexed so that $A(G) e_j=\lambda_j(G)e_j$ for all
$j\in \{1,\ldots,n\}$, and $e_1=\1_V$ (the constant $1$ function).


The following lemma becomes Lemma~\ref{lem:interpolated infty} when
$q=r=\infty$.

\begin{lemma}\label{lem:interpolated} Assume that $p,r> 1$ and $q>\frac{2p}{p-1}$. Let
$G=(V,E)$ be an $n$-vertex $d$-regular graph, and let
$e_1,\ldots,e_n$ be an orthonormal eigenbasis of the normalized
adjacency matrix $A(G)$. Then for every $\emptyset\neq S\subsetneq
V$ we have,
$$
\left|\frac{E_G(S,V\setminus S)}{\frac{d}{n}|S|(n-|S|)}-1\right|\le
M^{\frac{q}{p(q-2)}}\left(\sum_{i=2}^n|\lambda_i(G)|^p\right)^{\frac{1}{p}}\left(\frac{|S|^{\delta}(n-|S|)^{1-\delta}+
|S|^{1-\delta}(n-|S|)^{\delta}}{n}\right)^{\frac{1}{1-\delta}},
$$
where
\begin{equation}\label{eq:defs}
M\eqdef \left( \sum_{j=1}^n\|e_j\|_{L_r(V)}^q\right)^{1/q} \quad
\mathrm{and}\quad \delta\eqdef \frac{q(r-2)}{pr(q-2)+q(r-2)}.
\end{equation}
\end{lemma}
\begin{proof}
Consider the linear operator $T:\R^V\to \R^n$ given by
$$
T(x)\stackrel{\mathrm{def}}{=} \Big(\langle x,e_1\rangle,\langle
x,e_2\rangle,\ldots,\langle x,e_n\rangle\Big). $$ Since
$\{e_1,\ldots,e_n\}$ is an orthonormal basis of $L_2(V)$ we have
$\|T(x)\|_{\ell_2^n}=\|x\|_{L_2(V)}$ for all $x\in \R^V$. Moreover,
$$
\|T(x)\|_{\ell_q^n}=\left(\sum_{j=1}^n|\langle
x,e_j\rangle|^q\right)^{1/q}\le\left(
\sum_{j=1}^n\|e_j\|_{L_r(V)}^q\|x\|^q_{L_{\frac{r}{r-1}}(V)}\right)^{1/q}=M
\|x\|_{L_{\frac{r}{r-1}(V)}}.
$$
In other words, we have the operator norm bounds
$$
\|T\|_{L_2(V)\to \ell_2^n}=1\quad \mathrm{and}\quad
\|T\|_{L_{\frac{r}{r-1}}(V)\to \ell_q^n}\le M. $$

Recall that $\frac{2p}{p-1}\in (2,q)$, so we can define $\e\in
(0,1)$ by $\frac{p-1}{2p}=\frac{\e}{2}+\frac{1-\e}{q}$, i.e.,
$$
\e=1-\frac{q}{p(q-2)}. $$ If we then define $a>1$ via
$\frac{1}{a}=\frac{\e}{2}+\frac{(1-\e)(r-1)}{r}$, i.e.,
$$
a=\frac{2pr(q-2)}{pr(q-2)+q(r-2)}\stackrel{\eqref{eq:defs}}{=}2(1-\delta),
$$
then the Riesz-Thorin interpolation theorem
(see~\cite[Ch.~XII]{Zyg02}) asserts that for every $x\in \R^V$ we
have
\begin{equation}\label{eq:interpolated}
\|T(x)\|_{\ell_{\frac{2p}{p-1}}^n}\le
M^{1-\e}\|x\|_{L_a(V)}=M^{\frac{q}{p(q-2)}}\|x\|_{L_a(V)}.
\end{equation}

Fix $S\subseteq V$ and consider the function $x\in \R^V$ given by
$$
x\eqdef (n-|S|)\1_S-|S|\1_{V\setminus S}.
$$
Since $\langle x,e_1\rangle=\sum_{u\in V} x(u) =0$, the
bound~\eqref{eq:interpolated} becomes
\begin{equation*}\label{eq:plugged x}
\left(\sum_{j=2}^n \left|\langle
x,e_j\rangle\right|^{\frac{2p}{p-1}}\right)^{\frac{p-1}{2p}}\le
M^{\frac{q}{p(q-2)}}
\left(\frac{|S|(n-|S|)^a+(n-|S|)|S|^a}{n}\right)^{\frac{1}{a}}.
\end{equation*}
Hence,
\begin{multline*}
|\langle A(G)x,x\rangle|=\left|\sum_{j=2}^n \lambda_j(G)\langle
x,e_j\rangle^2\right|\le \left(\sum_{j=2}^n
|\lambda_j(G)|^p\right)^{\frac{1}{p}}\left(\sum_{j=2}^n
\left|\langle
x,e_j\rangle\right|^{\frac{2p}{p-1}}\right)^{\frac{p-1}{p}}\\
\le M^{\frac{2q}{p(q-2)}}\left(\sum_{j=2}^n
|\lambda_j(G)|^p\right)^{\frac{1}{p}}
\left(\frac{|S|(n-|S|)^a+(n-|S|)|S|^a}{n}\right)^{\frac{2}{a}}.
\end{multline*}
The required result now follows from the identity
\begin{equation*}
\langle A(G)x,x\rangle=|S|(n-|S|)-\frac{n}{d}E_G(S,V\setminus
S).\qedhere
\end{equation*}
\end{proof}

\section{Proof of Theorem~\ref{thm:our}}

Let $\Gamma$ be a group of cardinality $n$.
Recall that $R:\Gamma\to
GL(L_2(\Gamma))$ is the right regular representation, and given $g_1,\ldots,g_k\in \Gamma$, the matrix $A(g_1,\ldots,g_k)$ is defined as in~\eqref{eq:def A(g_i)}.

\begin{lemma}\label{lem:AR schatten}
There exists a universal constant $C\in (0,\infty)$ with the
following property. Fix $k,n\in \N$ and a group $\Gamma$ of
cardinality $n$. Let $g_1,\ldots,g_k\in \Gamma$ be chosen
independently and uniformly at random. Then with probability at
least $\frac12$, for every integer $p\ge 2$ we have
$$
\left\|\left(I-\frac{1}{n}J\right)A(g_1,\ldots,g_k)\right\|_{S_p}\le
Cn^{1/p}\sqrt{\frac{p}{k}}.
$$
\end{lemma}
\begin{proof}
Note that we
have for all $i\in \{1,\ldots,k\}$,
\begin{equation}\label{eq:exp AR}
\E\left[\left(I-\frac{J}{n}\right)\frac{R(g_i)+R(g_i^{-1})}{2}\right]=0.
\end{equation}
Moreover, because $MM^*=I-\frac{1}{n}J$, where
$M=\left(I-\frac{1}{n}J\right)R(g_i)$, for $p\ge 2$ we have
$\left\|\left(I-\frac{1}{n}J\right)R(g_i)\right\|_{S_p}=(n-1)^{1/p}$.
We therefore have the (point-wise) bound
\begin{equation}\label{eq:pointwise bound}
\left\|\left(I-\frac{1}{n}J\right)\frac{R(g_i)+R(g_i^{-1})}{2}\right\|_{S_p}
\le (n-1)^{1/p}.
\end{equation}
Theorem~\ref{thm:moment} now implies that for some universal
constant $c\in (0,\infty)$,
\begin{equation*}\label{eq:explicit moment bound}
\E\left[\left\|\left(I-\frac{1}{n}J\right)A(g_1,\ldots,g_k)\right\|_{S_p}^p\right]\le
\left(\frac{cp}{k}\right)^{p/2}n.
\end{equation*}
Hence,
$$
\E\left[\sum_{p=2}^\infty
\left(\frac{\left\|\left(I-\frac{1}{n}J\right)A(g_1,\ldots,g_k)\right\|_{S_p}}{2n^{1/p}\sqrt{cp/k}}\right)^p
\right]\le \sum_{p=2}^\infty \frac{1}{2^p}=\frac12.
$$
It follows from Markov's inequality that with probability at least
$\frac12$ we have
$$
\max_{p\in \N\cap [2,\infty)}
\left(\frac{\left\|\left(I-\frac{1}{n}J\right)A(g_1,\ldots,g_k)\right\|_{S_p}}{2n^{1/p}\sqrt{cp/k}}\right)^p\le
\sum_{p=2}^\infty
\left(\frac{\left\|\left(I-\frac{1}{n}J\right)A(g_1,\ldots,g_k)\right\|_{S_p}}{2n^{1/p}\sqrt{cp/k}}\right)^p\le
1.\qedhere
$$
\end{proof}

\begin{proof}[Proof of Theorem~\ref{thm:our}]
 $\Gamma$ is now Abelian, and therefore for every $g_1,\ldots,g_k\in \Gamma$,   the characters of $\Gamma$ are an orthonormal eigenbasis of $A(g_1,\ldots,g_k)$, consisting of functions whose absolute value is point-wise bounded by $1$. By Lemma~\ref{lem:interpolated} and Lemma~\ref{lem:AR schatten}, with probability at least $\frac12$ over i.i.d. uniform choice of $g_1,\ldots,g_k\in \Gamma$, if we let $G$ be the Alon-Roichman graph whose adjacency matrix is $A(g_1,\ldots,g_k)$, then every $S\subseteq \Gamma$ with $2\le |S|\le \frac{n}{2}$ satisfies
\begin{multline}\label{optimize p}
\left|\frac{E_G(S,V\setminus S)}{\frac{2k}{n}|S|(n-|S|)}-1\right|\lesssim \min_{p\in \N\cap [2,\infty)} \left( n^{\frac{1}{p}}\sqrt{\frac{p}{k}}\left(\frac{|S|^{\frac{1}{p+1}}(n-|S|)^{\frac{p}{p+1}}+
|S|^{\frac{p}{p+1}}(n-|S|)^{\frac{1}{p+1}}}{n}\right)^{\frac{p+1}{p}}\right)\\
\lesssim \min_{p\in \N\cap [2,\infty)}\left(|S|^{\frac{1}{p}}\sqrt{\frac{p}{k}}\right) \lesssim \sqrt{\frac{\log |S|}{k}},
\end{multline}
where int he last step of~\eqref{optimize p} we choose $p=2\lceil \log |S|\rceil$.
\end{proof}

\section{Proof of Theorem~\ref{thm:small k}}

If $k\le 1/\e^2$ then Theorem~\ref{thm:small k} is vacuous for $c>0$
small enough. Assuming $k\ge 1/\e^2$, denote $p=2\e^2 k\ge 2$.
Arguing analogously to~\eqref{optimize p}, we see that there exists
a universal constant $K\in (0,\infty)$ such that with probability at
least $\frac12$ over i.i.d. uniform choice of $g_1,\ldots,g_k\in
\Gamma$, if we let $G$ be the Alon-Roichman graph whose adjacency
matrix is $A(g_1,\ldots,g_k)$, then for all $S\subseteq \Gamma$ with
$2\le |S|\le \frac{n}{2}$,
$$
\left|\frac{E_G(S,V\setminus S)}{\frac{2k}{n}|S|(n-|S|)}-1\right|\le
K|S|^{\frac{1}{p}}\sqrt{\frac{p}{k}}= \e
K\sqrt{2}|S|^{\frac{1}{2\e^2k}}\le 2K\e,
$$
provided $|S|\le 2^{\e^2k}$.\qed


\bibliographystyle{abbrv}
\bibliography{martingale}

\end{document}